\documentclass[letterpaper, 10 pt, conference]{ieeeconf}  

\IEEEoverridecommandlockouts                              
\usepackage{tikz}
\usepackage{pgf,tikz}
\usepackage{cite}
\usepackage{amsmath,amssymb,amsfonts}
\usepackage{algorithmic}
\usepackage{graphicx}
\usepackage{textcomp}
\usepackage{xcolor}
\usepackage{float}
\usepackage{comment}
\usepackage{graphics} 
\usepackage{times} 
\usepackage{amsmath} 
\usepackage{amssymb}  

\newtheorem{theorem}{Theorem}  
\newtheorem{proposition}{Proposition}
\newtheorem{definition}{Definition}
\newtheorem{lemma}{Lemma}
\newtheorem{remark}{Remark}
\newtheorem{assumption}{Assumption}

\newcommand{\R}{\mathbb{R}}

\def\d{\mathrm{d}}

\title{\LARGE \bf
Sliding mode control for a class of linear infinite-dimensional systems}

\author{Isma\"ila Balogoun, Swann Marx and Franck Plestan
\thanks{The authors are with Nantes Université, École Centrale Nantes, CNRS, LS2N, UMR 6004, F-44000 Nantes, France (e-mail: \{ismaila.balogoun,swann.marx\}@ls2n.fr; franck.plestan@ec-nantes.fr).}
}    
\begin{document}

\maketitle

\begin{abstract}
This paper deals with the stabilization of a class of linear infinite-dimensional systems with unbounded control operators and subject to a boundary disturbance. We assume that there exists a linear feedback law that makes the origin of the closed-loop system globally asymptotically stable in the absence of disturbance. To achieve our objective, we follow a sliding mode strategy and we add another term to this controller  in order to reject the disturbance. We prove the existence of solutions to the closed-loop system and its global asymptotic stability, while making sure the disturbance is rejected.
\end{abstract}

\section{Introduction}
 This paper is concerned with the stabilization of a class of linear infinite-dimensional systems with unbounded control operators and subject to a boundary disturbance (see e.g,  \cite{bensoussan2007representation,curtain2020introduction,lasiecka1991differential} for a review on this class of system). To be more precise, we aim to design a sliding mode control (SMC) \cite{edwards1998sliding,shtessel2014sliding,utkin2013sliding} for the stabilization of boundary or pointwise control for linear partial differential equations (PDEs). We further propose a super-twisting control (STC), where, in contrast with SMC, the control is continuous.
 
 The boundary control of systems described by partial differential equations has received a lot of much since decades. It continues to be an important research focus today because its application in many important engineering systems is natural (see e.g., \cite{bastin2016stability}). Such a problem has been studied in \cite{rosier,russell1973unified,coron2004global,green2016boundary} in the controllability context, in \cite{slemrod1976stabilization,urquiza2005rapid,coron2021boundary,lucoron,zhou2013stabilization,komornik1997rapid} in
terms of stabilization.

In this paper, as mentioned earlier, we focus on the case where infinite-dimensional systems are subject to a disturbance. Therefore, we are not only interested in the stabilization, but also in the rejection of this disturbance. This might be interpreted as a regulation problem. In the case where the disturbance is constant, one can follow a proportional integral (PI) strategy, which is quite well-known in the linear finite-dimensional context, but which is still nowadays an active topic when dealing with PDEs (see e.g., \cite{balogoun2021iss,lhachemi2020pi,lhachemi2019pi,terrand2019adding,paunonen2010internal}). For more complicated cases, i.e. when the disturbance is time-dependant, one may apply the celebrated internal-model approach \cite{paunonen2010internal, deutscher2017output}, which consists, roughly speaking, in adding the dynamics of the disturbance in the loop of the controller. This method needs therefore the knowledge of the dynamics of the controller. Our strategy, based on SMC controllers, is in contrast with the latter one, since only the bound of the disturbance is needed, at the price of assuming that the disturbance matches with the control (i.e., the control and the disturbance are located at the same place).
 
SMC strategy has been proved to be efficient for robust control of nonlinear systems of ordinary differential equations (ODEs)\cite{shtessel2014sliding,utkin2013sliding,edwards1998sliding,young1999control}. Such controllers allow to force, thanks to discontinuous terms, the trajectories of the system to reach in a finite time a manifold, called the  sliding surface, and to evolve on it, this manifold being defined from control objectives. Basically, the design of the control is split into two steps: firstly, a sliding variable is selected such that, once this variable equals zero, global  asymptotic  stability  is  ensured;  secondly,  a discontinuous feedback-law is designed such that the trajectory reaches the sliding surface, that is defined thanks to the sliding variable. On this sliding surface, the disturbance is rejected. 
The generalization of the SMC procedure to the PDEs case is not new. In \cite{orlov1987sliding,orlov2000discontinuous}, a definition of equivalent control (which is the control  applied to the system after reaching the sliding surface, to ensure that the trajectories stays on the surface thereafter) for systems governed by semilinear differential equations in Banach spaces has been proposed.  One can refer also to \cite{levaggi2002infinite,levaggi2002sliding} where differential inclusions and viability theory are combined to design sliding mode controllers for semilinear differential equations in Banach spaces. We also mention the use of spectral reduction methods in \cite{orlov2004robust}. In the last decade, a backstepping strategy has been used to select  a sliding variable \cite{guo2012sliding,pisano2019combined,wang2015sliding,tang2014sliding,liu2015active}. We also refer to these recent papers \cite{liard2022boundary,balogoun2022}, in which the sliding variable is derived from  the gradient of some well-known Lyapunov functional in the hyperbolic context \cite[Section 2.1.2]{bastin2016stability}.  Note also that the SMC feedback-law is discontinuous, which creates chattering phenomena when implementing the control numerically. Therefore, in practical control cases, it is important to reduce this phenomena by providing continuous or smooth controller.

Based on second-order sliding mode techniques (see e.g, \cite[Chapter 4]{shtessel2014sliding}), the super twisting  algorithm has been developed for systems whose the sliding variable admits a relative degree (see \cite[Definition 1.6]{shtessel2014sliding}) equal to $1$. The essential feature of the super twisting control is to require  only the measurement of the sliding variable to guarantee the convergence in finite time to zero of the sliding variable and its derivative. Moreover, the super twisting feedback-law is continuous with respect to the state, and this drastically attenuates the chattering phenomenon.

In this paper, a  strategy different from the ones that have been mentioned earlier is proposed in order to design "classical" sliding mode controls and super-twisting sliding mode controls for general linear infinite-dimensional systems. The sliding variable is defined as the scalar product of the state and an eigenfunction of the adjoint operator of the closed-loop system  without disturbance. This requires measurement of the scalar product of the state with some function. Such a sliding variable allows to directly use well-known results on the stabilization of abstract infinite-dimensional systems with unbounded control operators in the absence of disturbance \cite{fattorini1968boundary,urquiza2005rapid,slemrod1976stabilization} together with well-known results about the finite-time convergence of the sliding variable in the context of the finite dimension \cite{moreno2012strict,polyakov2008lyapunov,utkin2013sliding}. In comparison with \cite{orlov1987sliding,orlov2000discontinuous,levaggi2002infinite,levaggi2002sliding}, the approach proposed in this document allows to define explicitly and systematically the sliding variable for a large class of linear infinite-dimensional systems.

This paper is organized as follows. Section \ref{sec_main} presents a class of linear infinite-dimensional system with an unbounded control operator, the sliding mode based control law, the super-twisting based control law and the main results of the paper. Section \ref{sec:proof} contains the proofs of the main results. Section \ref{sec:exemple} introduces an illustrative example. Finally, Section \ref{sec_conclusion} collects some remarks and introduces some future research lines to be followed.

Notation: Let $c\in \mathbb{C}$, $\mathfrak{Re}(c)$ (resp. $\mathfrak{Im}(c)$) denotes the real part (resp. the imaginary part) of $c$. The set of non-negative real numbers is denoted in this paper by $\mathbb R_+$.  When a function $f$ only depends on the time variable $t$ (resp. on the space variable $x$), its derivative is denoted by $\dot{f}$ (resp. $f^\prime$). Given any subset of $\mathbb{R}$ denoted by $\Omega$ ($\mathbb{R}_+$ or an interval, for instance), $L^p(\Omega)$ denotes the set of (Lebesgue) measurable functions $f$ such that, $\int_\Omega |f(x)|^p dx<+\infty$ when $p\neq +\infty$ and such that $\sup\mathrm{ess}_{x\in \Omega} |f(x)|<+\infty$ when $p=+\infty$. The associated norms are, for $p\neq +\infty$, $\Vert f\Vert_{L^p(\Omega)}^p:= \int_\Omega |f(x)|^p dx $ and, for $p=+\infty$, $\Vert f\Vert_{L^\infty(\Omega)}:= \sup\mathrm{ess}_{x\in \Omega} |f(x)$. For any $p\in [1,\infty]$, the Sobolev space $\mathcal{W}^{1,p}(\Omega)$ is defined by the set $\lbrace f \in L^p(\Omega)\mid f^\prime\in L^p(\Omega)\rbrace$. For  $m\geq 2$,  the Sobolev space $\mathcal{W}^{m,p}(\Omega)$ is defined by the set $\lbrace f \in\mathcal{W}^{m-1,p}(\Omega)\mid f^\prime\in \mathcal{W}^{m-1,p}(\Omega)\rbrace$. We also set $\mathcal{H}^p(\Omega)=\mathcal{W}^{m,2}(\Omega).$ We say that the function $f\in L^p_{loc}(\Omega)$ (resp. $f\in \mathcal{W}_{loc}^{m,p}(\Omega)$)  if the restriction  of $f\chi_K \in L^p(\Omega)$ (resp. $f\chi_K\in \mathcal{W}^{m,p}(\Omega)$) for every compact set $K$ contained in $\Omega$, where $\chi_K$ is the characteristic function.  Given two vector spaces $E$ and $F$, $\mathcal{L}(E,F)$ denotes the space of linear continuous applications from $E$ into $F$. If $E$ is a normed vector space, we denote by $\Vert \cdot\Vert_E$ the norm on $E$. We denote by $E'$ the dual space of $E$, that is, the space of all continuous linear functionals on $E$ and we denote by $ \langle \cdot,\cdot\rangle_{E,E'}$ the dual product on $E\times E'$. We denote by $C(E;F)$ the space of continuous functions from the  space $E$ to the space $F$. Throughout the paper, the field $\mathbb{K}$ is either $\mathbb{R}$ or $\mathbb{C}$.

\section{Main results}
\label{sec_main}
\subsection{Problem Statement}
Let $(H, \langle \cdot , \cdot \rangle_H) $  denotes a  Hilbert
space over the field $\mathbb{K}$ and the corresponding norm is denoted by $\Vert \cdot \Vert_H$. In this paper we are interested in the stabilization (at the origin) problem for the system 
\begin{equation}
\label{system_depart}
\left\{
\begin{split}
&\frac{\d}{\d t} z = A z + B(u+d),\\
&z(0)=z_0,
\end{split}
\right.
\end{equation}
where $z(t) \in H$ is the state, $u(t)\in  \mathbb{K}$ is the control input and $d(t)\in\mathbb{K}$  is an unknown disturbance.
In system \eqref{system_depart}, $A:D(A)\subseteq H\rightarrow H$ is a linear operator with $D(A)$ densely defined in $H$ and  $B\in \mathcal{L}(\mathbb{K},D(A^*)')$, with $A^*$ the adjoint operator of $A$. Our objective is to provide a design method so that system \eqref{system_depart} is globally stabilized despite the disturbance $d$. To do so, we will follow the sliding mode strategy. 

This strategy can be applied thanks to the following set of assumptions.
\begin{assumption}\label{assump1}
The following statements hold. 
\begin{itemize}
\item[(i)]  The operator $A:D(A)\subseteq H\rightarrow H$ generates a strongly continuous semigroup, that is denoted by $(\mathbb{T}(t))_{t\geq0}$. 
\item[(ii)] The operator $B$  is admissible\footnote{See e.g \cite[Section 4.2]{tucsnak2009observation}} for $(\mathbb{T}(t))_{t\geq0}$.
\item [(iii)]Let $T>0$. The pair $(A, B)$ is approximately controllable in $H$ in time $T$, i.e \begin{align}
    \label{observabilité}\notag
    &\forall \varphi \in D(A^*),\quad\forall t \in [0, T],\\ &B^*\mathbb{T^*}(T-t)\varphi=0 \Longrightarrow \varphi=0,
\end{align}

 where $B^*\in \mathcal{L}(D(A^*),\mathbb{K})$ is the adjoint operator of $B$.
\item[(iv)]There exists an operator $L:D(L)\to \mathbb{K}$ such that the operator \begin{equation}
\left\{
\begin{split}
&A_L = A +BL,\\
&D(A_L)=\{z\in D(L); (A+BL)z\in H \},
\end{split}
\right.
\end{equation} is the infinitesimal generator of a strongly continuous semigroup $(\mathbb{S}(t))_{t\geq0}$ on $H$  and the  origin of the following  system\begin{equation}
\label{system-stable}
\left\{
\begin{split}
&\frac{\d}{\d t} z = (A +BL)z,\\
&z(0)=z_0,
\end{split}
\right.
\end{equation} is globally asymptotically stable.
\end{itemize}
\end{assumption}
 Items (i) and (ii) allow to state the well-posedness of system \eqref{system_depart} in $H$ and Item (iii) refers to  a  controllability property of the system \eqref{system_depart}. This property is used in order to ensure finite-time stability of the sliding variable. Finally,  Item (iv) of  Assumption \ref{assump1} refers to a stabilizability property of system \eqref{system_depart}, needed to ensure that, without disturbance, the system can be stabilized.

The disturbance $d$ is not supposed to be known entirely, but we assume the knowledge of its bound. 
\begin{assumption}\label{assump2}
The unknown disturbance $d$ is supposed to be uniformly bounded measurable, i.e $\vert d(t)\vert\leq K_d$ for some $K_d>0$ and for all $t\geq 0$.
\end{assumption}
\begin{remark}
Note that Item (iv) of Assumption \ref{assump1} has been proven in \cite[Theorem 2.1]{urquiza2005rapid} in the case where the pair $(A,B)$ is exactly controllable in time $T$.
\end{remark}
Our goal is to find a state feedback control $u$ which allows to reject the disturbance and  to globally asymptotically stabilize the system \eqref{system_depart} around $0$.  Precisely, we are looking for a sliding surface on which the system \eqref{system_depart} becomes the system \eqref{system-stable} in a finite time. According to the item (iv) of Assumption \ref{assump1}, we know  that $0$ is globally asymptotically stable for the system \eqref{system-stable}. The next section will provide a definition of this sliding surface (and its related sliding variable), the associated sliding mode controllers and the associated super-twisting controllers.

\subsection{Sliding surface}
Let $\varphi \in D(A_L^*):=\{\varphi\in H\mid \exists c>0,\forall \phi\in D(A_L), \vert\langle \varphi, A_L(\phi)\rangle_H\vert\leq c\Vert\phi\Vert_H\}\subset D(A^*)$ is an eigenfunction of the adjoint operator of $A_L$ and $\lambda$ the eigenvalue associated with $\varphi$.
We introduce the following sliding surface $\Sigma$ 
$$\Sigma:=\left\{z \in H\mid  \langle \varphi, z\rangle_H=0\right\}.$$
Its related sliding variable
 $\sigma: \R_+\to \mathbb{K}$ is   defined by
\begin{equation}
 \label{sigma}
     \sigma(t):=\langle \varphi, z(t)\rangle_H
     \end{equation}
for any solution $z$ of \eqref{system_depart}.
This sliding variable represents the scalar product between the state and an eigenfunction of $A_L^*$. 

In this paper, we are interested in the design  of a sliding mode controller and  a super twisting controller. In the following section, we begin with the design of the sliding mode control.

\subsubsection{Sliding mode Control}
Since $\varphi$ is an eigenfunction which is different from $0$, then according to Item (iii) of Assumption \ref{assump1}, setting $t=T$, we obtain $B^*\varphi\neq 0$. Therefore, we can consider the  sliding mode controller $u$ defined by, for a.e $t\geq 0$,
 \begin{equation} \label{slidingcontrol}  u(t) = Lz(t)-\frac{1}{B^*\varphi}\bigg(\lambda\sigma(t)+ K\mathrm{sign}(\sigma(t))\bigg),  \end{equation}
 where  $\sigma$ is given  in \eqref{sigma}, $K$  is a positive constant that will be chosen later. Moreover,  the  set-valued function   $\mathrm{sign}$ is defined by
$$ \mathrm{sign}(s)=\left\{\begin{array}{ll}
\, \frac{s}{\vert s\vert} & \text{if } s\neq0,  \\
   \, [-1,1] & \text{if }  s=0.  \\
\end{array}\right.
$$Note that, since $B^*\in \mathcal{L}(D(A^*),\mathbb{K})$, then  $B^*\varphi$ is a scalar. Thus, we make the following  assumption about the constant $K$.
\begin{assumption}\label{assump3}
The constant $K$ is chosen such that  $\frac{K}{\vert B^*\varphi\vert} >K_d$.
\end{assumption}
Then, the closed$-$loop system \eqref{system_depart}$-$\eqref{slidingcontrol} can be written as
\begin{equation}
\label{system-final}
\left\{
\begin{split}
&\frac{\d}{\d t} z \in A_Lz+B\Bigg(d-\frac{1}{B^*\varphi}\bigg(\lambda\sigma(t)+ K\mathrm{sign}(\sigma(t))\bigg)\Bigg),\\
&z(0)=z_0.
\end{split}
\right.
\end{equation}
Formally, the derivative of $\sigma$ along the trajectory of \eqref{system-final} yields,  for all $t\geq 0$
\begin{align}\label{sliding mode equation}\notag
   \dot\sigma(t)&=\langle \varphi,\frac{\d}{\d t} z(t)\rangle_H\\\notag
   &\in\langle \varphi,A_Lz(t)\rangle_H\\\notag &+B^*\varphi\Bigg(d(t)-\frac{1}{B^*\varphi}\bigg(\lambda\sigma(t)+ K\mathrm{sign}(\sigma(t))\bigg)\Bigg)\\
   &\in \langle A_L^* \varphi,z(t)\rangle_H \\\notag&+B^*\varphi\Bigg(d(t)-\frac{1}{B^*\varphi}\bigg(\lambda\sigma(t)+ K\mathrm{sign}(\sigma(t))\bigg)\Bigg)\\\notag
   &\in \lambda\langle \varphi,z(t)\rangle_H \\\notag&+B^*\varphi\Bigg(d(t)-\frac{1}{B^*\varphi}\bigg(\lambda\sigma(t)+ K\mathrm{sign}(\sigma(t))\bigg)\Bigg)\\\notag
    \dot \sigma (t)&\in  B^*\varphi\bigg(d(t)-\frac{K}{B^*\varphi}  \mathrm{sign}(\sigma(t))\bigg)
\end{align} where, before the last equation, we used $A_L^* \varphi=\lambda\varphi$. Then, the following holds, for all $t\geq 0$
\begin{equation}\label{reach}
\begin{array}{ll}
    \dfrac{1}{2}\dfrac{d}{dt}\vert \sigma(t)\vert^2&= \mathfrak{Re}\big(\bar{\sigma}(t)\dot \sigma (t)\big)\\
    &=\mathfrak{Re}\bigg(\bar{\sigma}(t)B^*\varphi\big(d(t)-\frac{K}{B^*\varphi}  \mathrm{sign}(\sigma(t))\big)\bigg)\\
    \quad&\leq -(K-\vert B^*\varphi \vert K_d)\vert \sigma(t)\vert.
\end{array}
\end{equation}Therefore, separating variables and integrating inequality \eqref{reach} over the time interval $0\leq s\leq t$, we obtain
\begin{equation}
    \vert \sigma(t)\vert\leq \vert \sigma(0)\vert -(K-\vert B^*\varphi \vert K_d)t.
\end{equation}
Thus, there exists a finite time $t_r>0$, for which we know a bound that will be given later on, such that $\sigma(t)=0$ for any $t>t_r$. This means that the system \eqref{system-final} reaches  the sliding surface $\Sigma$ in finite time $t_r$ and remains on it. Since $\sigma(t)=0$ for any $t>t_r$, then $\dot\sigma(t)=0$ for any $t>t_r$. Thus, from \eqref{sliding mode equation}, we have $d(t)-\frac{K}{B^*\varphi}  \mathrm{sign}(\sigma(t)) = 0$ for any $t>t_r$. As a consequence, the system \eqref{system-final} can be rewritten as   \eqref{system-stable} on the sliding surface, which is globally asymptotically stable around $(0,0)$  from the item (iv) of Assumption \ref{assump1}. 

The next section focuses on the design of the super twisting control.

\subsubsection{Super twisting control}
In this section, we make the following assumption about the disturbance.
\begin{assumption}\label{assump4}
The disturbance $d(\cdot)$ is  globally Lipschitz over $\mathbb{R}_+$ and   there exists a known positive constant $C$ such that, for a.e $t\in\mathbb{R}_+$, \begin{equation}\label{bonne_d}
    \vert \dot{d}(t)\vert\leq C.
\end{equation} 
\end{assumption}
We assume that $\mathbb{K}=\mathbb{R}$. We do not treat the complex case, since we are not aware whether there exist super-twisting controllers for system whose state is in $\mathbb{C}$. We consider the  super twisting controller $u$ defined by, for all $t\geq 0$,
 \begin{equation}\label{super twisting control} 
\left\{
\begin{aligned}
& u(t) = Lz(t)\\& +\frac{1}{B^*\varphi}\bigg(-\lambda\sigma(t) -\alpha\vert\sigma(t)\vert^{\frac{1}{2}} \mathrm{sign}(\sigma(t))+v(t)\bigg),\\
&\dot v(t)\in-\beta\mathrm{sign}(\sigma(t)),
\end{aligned}
\right.
\end{equation}
 where  $\sigma$ is given  in \eqref{sigma}, $\alpha$ and $\beta$  are  positive constants which will be chosen later.
 
 Formally, the derivative of $\sigma$ along the trajectory of \eqref{system_depart}$-$\eqref{super twisting control} yields, for all $t\geq 0$
 \begin{align}\label{sliding mode equation super twisting}\notag
   \dot\sigma(t)&=\langle \varphi,\frac{\d}{\d t} z(t)\rangle_H\\\notag
   &=\langle \varphi,A_Lz(t)\rangle_H +B^*\varphi d(t)-\lambda\sigma(t)+v(t) \\\notag&-\alpha\vert\sigma(t)\vert^{\frac{1}{2}} \mathrm{sign}(\sigma(t))\\
   &= \lambda \sigma(t)-\lambda \sigma(t)+B^*\varphi d(t) +v(t)\\\notag&-\alpha\vert\sigma(t)\vert^{\frac{1}{2}} \mathrm{sign}(\sigma(t)) \\\notag
   & = -\alpha\vert\sigma(t)\vert^{\frac{1}{2}} \mathrm{sign}(\sigma(t))+v(t)+B^*\varphi d(t).
\end{align}Then, according to the following transformation \begin{equation}\label{transformation}
    w(t)=B^*\varphi d(t)+v(t),
\end{equation} we obtain 
\begin{equation}\label{super_twisting_equation}
\left\{
\begin{aligned}
&\dot{\sigma}(t) =-\alpha\vert \sigma(t)\vert^{\frac{1}{2}}\mathrm{sign}(\sigma(t))+w(t),\\
&\dot{w}(t) \in B^*\varphi \dot{d}(t) -\beta\mathrm{sign}(\sigma(t)).
\end{aligned}
\right.
\end{equation}
From  \cite[Theorem 1]{seeber2017stability}, all trajectories of \eqref{super_twisting_equation} converge to zero in finite time.
\begin{proposition}(\cite[Theorem 1]{seeber2017stability})\label{super_twisting_proposition}
Assuming that \begin{equation}\label{parametre}
     \beta>\vert B^*\varphi \vert C\quad\mbox{ and } \alpha> \sqrt{\beta +\vert B^*\varphi \vert C},
 \end{equation}there exists a finite time $t_r>0$ 
such that $\sigma(t)=0$ and $w(t)=0$ for any $t>t_r$.
\end{proposition}
 Then, the closed$-$loop system \eqref{system_depart}$-$\eqref{super twisting control} can be written as
\begin{equation}
\label{system-final_super_twisting}
\left\{
\begin{split}
&\frac{\d}{\d t} z= A_Lz+B\Bigg(\frac{1}{B^*\varphi}\bigg(-\lambda\sigma(t) -\alpha\vert\sigma(t)\vert^{\frac{1}{2}} \mathrm{sign}(\sigma(t))\\&+w(t)\bigg)\Bigg),\\
&\dot{w}(t) \in B^*\varphi \dot{d}(t) -\beta\mathrm{sign}(\sigma(t)),\\
&z(0)=z_0\in H, w(0)=w_0\in\mathbb{R}.
\end{split}
\right.
\end{equation}

\subsubsection{Main results} 
The equations \eqref{sliding mode equation} and \eqref{super_twisting_equation} are understood in the sense of Filippov \cite[Chapter 2]{filippov2013differential}, we recall the definition.  
\begin{definition}A Filippov solution of \eqref{sliding mode equation} (resp. of \eqref{super_twisting_equation}) is an absolutely continuous map that satisfies \eqref{sliding mode equation} (resp. \eqref{super_twisting_equation}) for almost all $t\geq 0$.\end{definition}
The solutions of \eqref{system-final} are understood in the sense of the following definition.
\begin{definition}\label{def:solution}
Let $z_0\in H$. We say that the map $z:[0,\infty)\to H$ is a mild  solution of \eqref{system-final}, if $z\in C([0,\infty);H)\cap \mathcal{H}^1_{loc}([0,\infty);D(A^*)')$ such that, for all $t\in [0,\infty)$, \begin{equation}\label{mildsolu_abstrait}
        z(t)=\mathbb{S}(t)z_0+\int^t_0 \mathbb{S}(t-s)Bh(s)ds,
    \end{equation} where $(\mathbb{S}(t))_{t\geq0}$ is the strongly continuous semigroup generated by the operator $A_L$ and  $h:[0,\infty)\to \mathbb{K}$ is in $ L^2_{loc}([0,\infty);\mathbb{K})$ and satisfies, 
 for a.e $t\geq 0$,
\begin{equation}\label{v_inclusion}
    h(t)\in -\frac{1}{B^*\varphi}\bigg(\lambda\sigma(t)+ K\mathrm{sign}(\sigma(t))\bigg)+ d(t)
\end{equation}
 with $\sigma$ given in \eqref{sigma}.
\end{definition}
The following definition indicates how  the solutions of \eqref{system-final_super_twisting} are understood. 
\begin{definition}\label{def:solution1}
Let $z_0\in H$ and $w_0\in\mathbb{R}$. We say that the map $z:[0,\infty)\to H$ and $w:[0,\infty)\to \mathbb{R}$ is a mild  solution of \eqref{system-final_super_twisting}, if $z\in C([0,\infty);H)\cap \mathcal{H}^1_{loc}([0,\infty);D(A^*)')$ and $w$ is absolutely continuous such that, \begin{equation}\label{mildsolu_abstrait_twisting}
       \mbox{for all }t\in [0,\infty),\quad z(t)=\mathbb{S}(t)z_0+\int^t_0 \mathbb{S}(t-s)B\omega(s)ds
    \end{equation}  and 
 
\begin{equation}\label{w_inclusion}
 \mbox{for a.e }t\in [0,\infty),\quad \dot{w}(t) \in B^*\varphi \dot{d}(t) -\beta\mathrm{sign}(\sigma(t)),
\end{equation}where $(\mathbb{S}(t))_{t\geq0}$ is the strongly continuous semigroup generated by the operator $A_L$ and \begin{equation}\label{omega}
    \omega(t)= \frac{1}{B^*\varphi}\bigg(-\lambda\sigma(t) -\alpha\vert\sigma(t)\vert^{\frac{1}{2}} \mathrm{sign}(\sigma(t))+w(t)\bigg)
\end{equation} 
 with $\sigma$ given in \eqref{sigma}.
\end{definition}
Note that, Definition \ref{def:solution} and Definition \ref{def:solution1} are based on the concept of mild solution\footnote{See e.g \cite[Definition 4.1.5]{tucsnak2009observation}}.

Before presenting the  results of this paper, we present the following definition of  the equilibrium point of systems \eqref{system-final} and \eqref{system-final_super_twisting}. 

\begin{definition}
\begin{enumerate}
    \item We say that $\tilde{z}\in H$   is an equilibrium point of system \eqref{system-final} , if $\tilde z\in D(A_L)$   and there exists $\tilde z^* \in [-K_d,K_d]-\frac{K}{B^*\varphi}  \mathrm{sign}(\langle \varphi, \tilde z\rangle_H)$   such that 
\begin{equation}\label{equilib_sliding_mode}
    A_L\tilde z+B\tilde z^*=0.
\end{equation}
\item We say that  $(\tilde{z},\tilde{w})\in H\times\mathbb{R}$ is an equilibrium point of system  \eqref{system-final_super_twisting}, if  $(\tilde z,\tilde w)\in D(A_L)\times\mathbb{R}$  and there exists  $\tilde z^* \in B^*\varphi [-C,C] -\beta\mathrm{sign}\big(\langle \varphi, \tilde z\rangle_H\big)$ such that
\begin{align}\label{equilib_super_twisting1}\notag
    A_L\tilde z+B\Bigg(\frac{1}{B^*\varphi}\bigg(-\lambda\langle \varphi, \tilde z\rangle_H&+\tilde w \\-\alpha\vert\langle \varphi, \tilde z\rangle_H\vert^{\frac{1}{2}} \mathrm{sign}\big(\langle \varphi, \tilde z\rangle_H\big)\bigg)\Bigg)=0
\end{align} and
\begin{equation}
\label{equilib_super_twisting2}
  \tilde z^* =0.
\end{equation}
\end{enumerate}
\end{definition}
\begin{remark}
One can check that $0\in H$ (resp. $(0,0)\in H\times \R$) is the unique equilibrium point of \eqref{system-final} (resp. \eqref{system-final_super_twisting}). 
\end{remark}

 The main results of this paper can be formulated as follows:
 
\begin{theorem}[Existence of  solutions]\label{main}
\begin{enumerate}
\item Assume that Assumption \ref{assump1}, Assumption \ref{assump2} and Assumption \ref{assump3}  are satisfied. For any initial condition $z_0\in H$, the system \eqref{system-final}  admits a mild solution.
\item Assume that Assumption \ref{assump1}, Assumption \ref{assump4} and Equation \eqref{parametre} are satisfied. For any initial condition  $z_0\in H$ and $w_0\in\R$,  the system \eqref{system-final_super_twisting} admits a mild solution.
\end{enumerate}
\end{theorem}
The next result of this paper is stated as follows:
\begin{theorem}[Global asymptotic stability]\label{Global asymptotic}
\begin{enumerate}
    \item Assume that Assumption \ref{assump1}, Assumption \ref{assump2} and Assumption \ref{assump3}  are satisfied. For any initial condition $z_0\in H$, $0\in H$  is globally asymptotically stable for \eqref{system-final}.
    \item  Assume that Assumption \ref{assump1}, Assumption \ref{assump4} and Equation \eqref{parametre} are satisfied. For any initial condition  $(z_0, w_0)\in H\times\R$,    $\begin{pmatrix}0\\0\end{pmatrix}\in H\times\R$ is globally asymptotically stable for \eqref{system-final_super_twisting}.
\end{enumerate}
\end{theorem} 

\section{Proof of Theorem \ref{main} and Theorem \ref{Global asymptotic} }\label{sec:proof}
\subsection{Proof of Theorem \ref{main}}
The proof of Theorem \ref{main} is divided into two parts. In the first part, the proof of the  Theorem \ref{main} is presented in the case of system \eqref{system-final}. The second part deals with the proof of Theorem \ref{main} in the case of system \eqref{system-final_super_twisting}.

Let us start the proof of the first part.
\subsubsection{Sliding mode control}
We consider the following ODE
\begin{equation} \label{ODE}
\left\{
\begin{aligned}
&\dot \gamma (t)\in  B^*\varphi\bigg(d-\frac{K}{B^*\varphi}  \mathrm{sign}(\gamma(t))\bigg), &t\in\R_+, \\
& \gamma (0)=  \gamma_0\in \R, \\
\end{aligned}
\right.
\end{equation}
 The system  \eqref{ODE} is understood in the sense of Filippov \cite{filippov2013differential}.  In the next lemma, we state that there exists a unique solution to \eqref{ODE} and that \eqref{ODE} is stabilized in finite-time.
 
\begin{lemma} \label{lem:ODE}
Assume that Assumption \ref{assump2} hold. Then, the ODE \eqref{ODE} admits a unique Filippov solution. Moreover,  there exists $t_r>0$ such that,  for any Filippov solution $\gamma$ of \eqref{ODE},   $$ \gamma(t)=0,\, \forall \,t\geq t_r,$$  with $$t_r\leq\frac{\vert \gamma(0)\vert }{K-K_d\vert B^*\varphi \vert}.$$
\end{lemma} Lemma \ref{lem:ODE} is an immediate consequence of the general Filippov theory \cite[Chapter 2]{filippov2013differential} (for the real case), \cite[Theorem 2.8]{wang2018generalized} (for the complex case), when  applied to the particular case of \eqref{ODE}. Finite-time stability can be deduced easily by Lyapunov arguments (given in Section \ref{sec_main}).

\ \\
Let $\gamma$ be the Filippov solution of \eqref{ODE} with initial condition $\gamma(0)=\langle \varphi, z_0\rangle_H$. We consider the following system \begin{equation}
\label{system-finall}
\left\{
\begin{split}
&\frac{\d}{\d t} \phi= A_L\phi+\frac{1}{B^*\varphi} B(\dot\gamma-\lambda\gamma),\\
&\phi(0)=\phi_0\in H.
\end{split}
\right.
\end{equation}
If $B$ is an admissible operator for $\mathbb{S}$ and $\dot\gamma-\lambda\gamma \in L^2_{loc}([0,\infty);\mathbb{K})$, then  system \eqref{system-finall} admits a  unique  mild solution, where $(\mathbb{S}(t))_{t\geq 0}$  is the strongly continuous semigroup associated  with the operator $A_L$. This is what we will prove in the next Lemma, which says that there exists a unique solution in the sense of \cite[Definition 4.1.5]{tucsnak2009observation}.
\begin{lemma}
\label{lembien}
For all $ \phi_0\in H$, the system \eqref{system-finall} admits a unique mild solution $\phi\in C([0,\infty);H)\cap \mathcal{H}^1_{loc}([0,\infty);D(A^*)')$.  
\end{lemma}
\begin{proof}
Let $\gamma$ be a Filippov solution of \eqref{ODE}. Then, according to Lemma \ref{lem:ODE}, $\gamma$ is absolutely continuous. Moreover, $\dot\gamma$ is 
bounded and measurable according to Assumption \ref{assump2}. Thus, we have  $\dot\gamma-\lambda\gamma \in L^2_{loc}([0,\infty);\mathbb{K})$. On the other hand, according to the item (ii) of Assumption \ref{assump1}, $B$ is admissible for $(\mathbb{T}(t))_{t\geq 0}$, then according to \cite[Proposition 4.2]{hansen1997new}, $B$ is an admissible control operator for $(\mathbb{S}(t))_{t\geq 0}$. Then, according to \cite[Proposition 4.2.5]{tucsnak2009observation}, the statement of Lemma \ref{lembien}  holds, achieving the proof.
\end{proof}
Now, the  aim is to prove that the mild solution $\phi$  to \eqref{system-finall} with initial condition $z_0$ is a mild solution to \eqref{system-final}. To that end, we will show that  the following  function \begin{equation} \label{def:y}y(t)=\langle \varphi, \phi(t)\rangle_H,\end{equation} with  $\phi$ the solution of \eqref{system-finall}, is equal to $\gamma$, for any $t>0$.
 \begin{lemma} \label{Caratheodry}
For all $ z_0\in H$, $y$ is a Carath\'eodory solution to 
\begin{equation} \label{EDOcool} \left\{\begin{array}{l}
\dot{y}(t)=\lambda y +\dot \gamma(t)-\lambda \gamma(t), \quad \mbox{for a.e } t\geq 0,  \\
y(0)=\langle \varphi, z_0\rangle_H
\end{array}\right.
\end{equation} i.e $y$ is an absolutely continuous map such that, for all $t\geq 0$  \begin{equation}
y(t) -y(0)= \int_0^t\big(\lambda y(s)+\dot\gamma(s)-\lambda\gamma(s)\big)ds.
\end{equation}
\end{lemma}
\begin{proof}
Let $\phi$ be the mild solution of \eqref{system-finall}. Since $\varphi \in D(A^*_L)$, and using Item (ii) of Assumption \ref{assump1}, then according to \cite[Remark 4.2.6]{tucsnak2009observation}, we obtain for that, every $t\geq 0$,
\begin{align}\label{def:dot_y}
\notag
     \langle \varphi, \phi(t) -z_0\rangle_H&=\int_0^t\bigg( \langle A_L^*\varphi,\phi(s)\rangle_H+\frac{1}{B^*\varphi}B^*\varphi(\dot\gamma(s)\\\notag&-\lambda\gamma(s))\bigg)ds\\
    &=\int_0^t\bigg( \lambda\langle \varphi,\phi(s)\rangle_H+\dot\gamma(s)-\lambda\gamma(s)\bigg)ds,
\end{align} because $A_L^*\varphi=\lambda\varphi$. Then, using \eqref{def:y}, one has, for all $t\geq 0$,
\begin{equation}
y(t) -y(0)= \int_0^t\big(\lambda y(s)+\dot\gamma(s)-\lambda\gamma(s)\big)ds.
\end{equation}
This concludes the proof.
\end{proof}
We introduce the function  $g$ defined by $g(t)=y(t)-\gamma(t)$. From \eqref{ODE} and \eqref{EDOcool} with $\gamma(0)=\langle \varphi, z_0\rangle_H$, $g$ is solution of 
\begin{equation} \label{EDOsimple}
\left\{\begin{array}{l}
\dot g(t)=\lambda g(t)\\
g(0)=0
\end{array}\right.
\end{equation} 
Thus, for any $t\in \R$, $g(t)=0$. By definition of $g$, we deduce that, for any $t\in \R$, $y(t)=\gamma(t)$. Therefore, according to \eqref{ODE} we have, for a.e $t\geq 0$, \begin{equation}\label{inclusion}
   \frac{1}{B^*\varphi} \dot \gamma(t) \in -\frac{K}{B^*\varphi}\mathrm{sign}(y(t)) + d(t).
\end{equation}
Thus, according to Lemma \ref{lembien} and \eqref{inclusion}, $\phi$ satisfies   Definition \ref{def:solution}. Then, we conclude that, for any Filippov solution $\gamma$ of \eqref{ODE} with initial condition $\gamma(0)=\langle \varphi, z_0\rangle_H$, the associated mild solution $\phi$ of \eqref{system-finall} is a  mild solution of \eqref{system-final}. This concludes the proof of Theorem \ref{main} in the case of system \eqref{system-final}.  \hfill$\Box$
\subsubsection{Super twisting control}

Let $z_0\in H$, $w_0\in\R$. Consider the following ODE

\begin{equation} \label{ODE1}
\left\{
\begin{aligned}
&\dot \rho (t)= -\alpha\vert \rho(t)\vert^{\frac{1}{2}}\mathrm{sign}(\rho(t))+\eta(t), &t\in\R_+, \\
&\dot\eta(t) \in  B^*\varphi \dot{d}(t) -\beta\mathrm{sign}(\rho(t)), & t\in \R_+\\
& \rho (0)=\rho_0, \eta(0)=w_0. \\
\end{aligned}
\right.
\end{equation}

The system  \eqref{ODE1} is understood in the sense of Filippov \cite{filippov2013differential}. In the next lemma, we state that there exists a solution to \eqref{ODE1}. 
\begin{lemma}
Assume that \eqref{parametre}  holds. Then, there exists an absolutely continuous map  $(\rho,\eta)$  that satisfies \eqref{ODE1} for almost every $t\geq 0$.  
\label{edosolution}
\end{lemma}
\begin{proof}  We  consider the function $f:\R^2\to \R^2$ defined by \begin{equation}
    f(\rho,\eta)=\left\{\begin{array}{ll}
\, f^+(\rho,\eta)=(-\alpha\sqrt{\rho}+\eta,-\beta)  & \text{if } \rho>0,  \\

\,  f^-(\rho,\eta)=(\alpha\sqrt{-\rho}+\eta,\beta)& \text{if } \rho<0 \\
\end{array}\right.
\end{equation}and let $F_d:(\rho,\eta)\in \R^2\mapsto F_d(\rho,\eta)$ be the set-valued map defined by\begin{align}\notag
    F_d(\rho,\eta)&=\bar B(0,\vert B^*\varphi \vert C)\\&+\left\{\begin{array}{ll}
\, \{f(\rho,\eta)\}  & \text{if } \rho\neq 0,  \\
\, \overline{\text{conv}}\{f^+(\rho,\eta),f^-(\rho,\eta)\} & \text{if } \rho=0 
\end{array}\right.
\end{align}
where $\bar B(0,\vert B^*\varphi \vert C)$ is a closed ball of $\R^2$ centered at $0$ and of radius $\vert B^*\varphi \vert C$.
 Since $f$ is continuous on $\R\setminus\{0\}\times\R$, then the function $F_d$ is non$-$empty, compact, convex and upper semi$-$continuous. Then according to \cite[Theorem 3.6]{bernuau2014homogeneity}, there exists at least one solution of the differential inclusion \begin{equation}
    \dot \zeta\in F_d(\zeta)
\end{equation} where $\zeta=(\rho,\eta)$. 
Since $F_d$ is the Filippov’s construction (as in \cite[Chapter 2]{filippov2013differential})  associated with \eqref{ODE1}, then, there exists an absolutely continuous map that satisfies \eqref{ODE1} for almost every $t\geq 0$, concluding therefore the proof.
\end{proof}

Let  $(\rho,\eta)$ be a solution of \eqref{ODE1} with initial condition $\rho(0)=\langle \varphi, z_0\rangle_H$. We consider the following system \begin{equation}
\label{system-finall1}
\left\{
\begin{split}
&\frac{\d}{\d t} \psi= A_L\psi+\frac{1}{B^*\varphi} B(\dot\rho-\lambda\rho),\\
&\psi(0)=z_0\in H.
\end{split}
\right.
\end{equation}
Since $\rho$ and $\eta$ are continuous then, according to the first line of \eqref{ODE1}, we  deduce that $\dot \rho$ is also continuous. Moreover, since $\rho$ and $\dot\rho$ are continuous,  then $\dot\rho-\lambda\rho \in  L^2_{loc}([0,\infty);\R)$. Thus, according to Lemma \ref{lembien}, the system \eqref{system-finall1} admits a unique mild solution $\psi\in C([0,\infty);H)\cap \mathcal{H}^1_{loc}([0,\infty);D(A^*)')~$. 

 As in the previous case, the  aim is now to prove that the solution $(\psi,\eta)$ is a mild solution  to \eqref{system-final_super_twisting}.
 For this purpose, we are going to show that the following function \begin{equation} \label{def:theta}\mathcal{\theta}(t)=\langle \varphi, \psi(t)\rangle_H,\end{equation} with  $\psi$ the solution of \eqref{system-finall1}, is equal to $\rho$ for any $t>0$.
 \begin{lemma} \label{Caratheodry1}
For all $ z_0\in H$, $\theta$ is a Carath\'eodory solution of 
\begin{equation} \label{EDOcool1} \left\{\begin{array}{l}
\dot{\theta}(t)=\lambda \theta +\dot \rho(t)-\lambda \rho(t), \quad t\geq 0,  \\
\theta(0)=\langle \varphi, z_0\rangle_H.
\end{array}\right.
\end{equation}
\end{lemma}
The proof of Lemma \ref{Caratheodry1} is similar to the proof of Lemma \ref{Caratheodry}. We therefore omit the proof of Lemma \ref{Caratheodry1}.

Now, according to Lemma \ref{edosolution} $\theta$ is absolutely continuous map. Moreover, if we  set $\kappa=\theta-\rho$, then $\kappa$ satisfies \eqref{EDOsimple}. Thus, $\kappa (t)=0$ for all $t\in \R$. This mean that,  $\theta(t)=\rho(t)$ for any $t\in \R$. As a consequence, according to \eqref{ODE1}, we have, for a.e $t\geq 0$, \begin{equation}\label{inclusion_eta}
  \dot\eta(t) \in  B^*\varphi \dot{d}(t) -\beta\mathrm{sign}(\theta(t)).
\end{equation}
Thus, according to Lemma \ref{edosolution} and \eqref{inclusion_eta}, $\eta$ is absolutely continuous map and satisfies \eqref{w_inclusion}. Therefore, $(\psi,\eta)$ satisfies  the  Definition \ref{def:solution1}. This, mean that $(\psi,\eta)$ is a mild solution of \eqref{system-final_super_twisting}. This concludes the proof of Theorem \ref{main} in the case of system \eqref{system-final_super_twisting}.

\hfill$\Box$

\subsection{Proof of Theorem \ref{Global asymptotic}}
Like the proof of Theorem \ref{main}, the proof of Theorem \ref{Global asymptotic} is divided into two parts. In the first part, the proof of the  Theorem \ref{Global asymptotic} is presented in the case of the system \eqref{system-final}. The second part deals with the proof of Theorem \ref{Global asymptotic} in the case of system \eqref{system-final_super_twisting}.

Let us start the proof of the first part.
\subsubsection{Sliding-mode control}
Let us consider $z$ a mild solution of  \eqref{system-final} with initial condition $z_0\in H$. Then, according Definition \ref{def:solution}, there exists $h\in L^2_{loc}([0,\infty);\mathbb{K})$ such that  $h$ satisfies \eqref{v_inclusion} and $z$ satisfies \eqref{mildsolu_abstrait}.
Therefore, since $\varphi \in D(A^*_L)$, and using Item (ii) of Assumption \ref{assump1}, then according to \cite[Remark 4.2.6]{tucsnak2009observation}, $z$ satisfies, for  every $t\geq 0$,
\begin{align}
\notag
     \langle \varphi, z(t) -z_0\rangle_H&=\int_0^t\bigg( \langle A_L^*\varphi,z(s)\rangle_H+B^*\varphi h(s)\bigg)ds\\
    &=\int_0^t\bigg( \lambda\langle \varphi,z(s)\rangle_H+B^*\varphi h(s)\bigg)ds, 
\end{align} because $A_L^*\varphi=\lambda\varphi$.
Using \eqref{sigma}, one has, for every $t\geq 0$,
\begin{equation} \label{sigma_cathe}
\sigma(t) -\sigma(0)= \int_0^t\big(\lambda \sigma(s)+B^*\varphi h(s)\big)ds.
\end{equation}
As a consequence, $\sigma$ defined in \eqref{sigma} is a Carath\'eodory solution to  
\begin{equation} \label{ODE_sigma}
\left\{
\begin{aligned}
&\dot \sigma (t)= \lambda \sigma(t)+B^*\varphi h(t), \\
& \sigma (0)=   \langle \varphi, z_0\rangle_H. \\
\end{aligned}
\right.
\end{equation}
Since $h\in -\frac{1}{B^*\varphi}\bigg(\lambda\sigma+ K\mathrm{sign}(\sigma)\bigg)+d$, then  $\sigma$ is a Filippov solution of \eqref{ODE} with initial condition $ \langle \varphi, z_0\rangle_H$. From Lemma \ref{lem:ODE}, there exists a finite time $t_r$  such that 
 $$  \sigma(t)=0 \text{ for any } t> t_r.$$ Therefore, $\dot\sigma(t)=0 \text{ for any } t> t_r$. As a consequence, from \eqref{ODE_sigma}, for any $t>t_r$, $h(t)=0$. Thus, for any $t>t_r$, the system \eqref{system-final} is equivalent to the system \eqref{system-stable}  and hence is asymptotically stable in $H$ from the item (iv) of Assumption \ref{assump1}. Therefore, to conclude the proof of Theorem \ref{Global asymptotic} in the case of system \eqref{system-final}, it is just necessary to prove the Lyapunov stability of the system \eqref{system-final} over the time interval $[0,t_r]$. For this purpose, we consider $z$ a mild solution of  \eqref{system-final} with initial condition $z_0\in H$ on the interval $[0,t_r]$. Then, using the Definition \ref{def:solution}, there exists $C_0>0$ such that, for all $t\in[0,t_r]$, we have 
 \begin{equation}\label{inegalite_stabilite}\Vert z(t) \Vert_{H}\leq C_0\Vert z_0 \Vert_{H}+
    \left\Vert \int^t_0 \mathbb{S}(t-s)Bh(s)ds\right\Vert_{H}.
\end{equation}
Since $(\mathbb{S}(t))_{t\geq0}$ is exponentially stable and $B$ is an admissible operator for $(\mathbb{S}(t))_{t\geq0}$, then according to \cite[Proposition 4.3.3]{tucsnak2009observation},  there exists $C_1>0$ independent of $t_r$  such that, for all $t\in[0,t_r]$
\begin{equation}\label{inegalite_stabilite_1}
\Vert z(t) \Vert_{H}\leq C_1\bigg(\Vert z_0 \Vert_{H}+
    \Vert h\Vert_{L^2(0,t_r)}\bigg).
\end{equation}Moreover, since $h\in -\frac{1}{B^*\varphi}\big(\lambda\sigma+ K\mathrm{sign}(\sigma)\big)+d$, then according to Assumption \ref{assump2}, $h$ is bounded. Therefore, there exists $C_2>0$ such that
\begin{equation}\label{v_borne}
     \Vert h\Vert_{L^2(0,t_r)}\leq C_2 t_r^{\frac{1}{2}}.
\end{equation}Moreover, according to Lemma\ref{lem:ODE}, $t_r\leq\frac{\vert \langle \varphi, z_0\rangle_H\vert }{K-\vert B^*\varphi \vert\Vert d\Vert_{L^\infty(\mathbb{R}_+)}}$. Thus, using Cauchy-Schwarz's inequality, we have
\begin{equation}\label{tr}
    t_r\leq\frac{\Vert\varphi\Vert_H}{K-\vert B^*\varphi \vert\Vert d\Vert_{L^\infty(\mathbb{R}_+)}} \Vert z_0\Vert_H.
\end{equation}As a consequence, according to \eqref{inegalite_stabilite_1}, \eqref{v_borne} and \eqref{tr}, there exists $C_3>0$ (independent of $t_r$) such that for all $t\in [0,t_r]$,
\begin{equation}\label{inegalite_stabilite_2}
\Vert z(t) \Vert_{H}\leq C_3\bigg(\Vert z_0 \Vert_{H}+
   \sqrt{\Vert z_0 \Vert_{H}}\bigg).
\end{equation}According to \cite[Definition 2.3]{prieur}, this concludes the proof of Lyapunov stability of the system \eqref{system-final} over the time interval $[0,t_r]$. \hfill $\Box$

\begin{remark}
In contrast with many stabilization techniques, we do not need here to compute time-derivative of Lyapunov functionals for the infinite-dimensional system. More precisely, classical techniques rely on the existence of strong solutions for which on computes time derivative of a suitable Lyapunov functional, and one concludes then on the stability for weak solution by a density argument. 
\end{remark}

\subsubsection{Super-twisting control}
Let us consider $(z,w)$ a mild solution of  \eqref{system-final_super_twisting} with initial condition $(z_0,w_0)\in H\times \R$. Then, according Definition \ref{def:solution1}, there exists $\tilde{w}\in L^1_{loc}([0,\infty);\R)$ with $\tilde{w}(t)\in\mathrm{sign}(\sigma(t))$  such that, for a.e $t\geq 0$,   $\dot w(t)=B^*\varphi d(t)-\beta\tilde{w}$ and $z$ satisfies \eqref{mildsolu_abstrait_twisting}. Replacing $h$ by $\omega$ in \eqref{sigma_cathe}, then $\sigma$ satisties \eqref{sigma_cathe}.
Then, according to  \eqref{omega}, \eqref{sigma_cathe} we obtain, for a.e $t\in [0,T]$
\begin{equation}\label{preuve_systemS}
\left\{
\begin{aligned}
&\dot{\sigma}(t) =-\alpha\vert \sigma(t)\vert^{\frac{1}{2}}\mathrm{sign}(\sigma(t))+w(t),\\
&\dot{w}(t) =B^*\varphi\dot{d}(t) -\beta \tilde{w}(t).
\end{aligned}
\right.
\end{equation}Since $\tilde{w}\in \mathrm{sign}(\sigma(t)$, then $(\sigma,w)$ is a Filippov solution of \eqref{ODE1} with initial condition $\left( \langle \varphi, z_0\rangle_H, w_0\right)$. According to Proposition \ref{super_twisting_proposition}, there exists a finite time such that  $$\sigma(t)=0\mbox{ and } w(t)=0$$ for any $t>t_r$. Then, for any $t>t_r$, the solution $z$ to  system \eqref{system-final_super_twisting} is solution to system \eqref{system-stable}  and hence is asymptotically stable in $H$ from Item (iv) of Assumption \ref{assump1}. Therefore, as in the previous part, we just need  to prove the Lyapunov stability of the system \eqref{system-final_super_twisting} over the time interval $[0,t_r]$ to conclude the proof of Theorem \ref{Global asymptotic}.
For this purpose, we consider  $(z,w)$ a mild solution of  \eqref{system-final_super_twisting} with initial condition $(z_0,w_0)\in H\times \R$ on the interval $[0,t_r]$. Then, like in the previous part, using  Definition \ref{def:solution1}, there exists $C_0>0$ such that, for all $t\in[0,t_r]$, we have 
\begin{equation}\label{inegalite_stabilite_twisting}
\Vert z(t) \Vert_{H}\leq C_0\bigg(\Vert z_0 \Vert_{H}+
    \Vert \omega\Vert_{L^2(0,t_r)}\bigg).
\end{equation}
Since  $z$ is  continuous on $[0,t_r]$ , then, according to \eqref{sigma}, $\sigma$ is also continuous. Therefore, $\sigma$ is bounded  on $[0,t_r]$. Moreover, $w$ is an absolutely continuous map. Thus, $w$ is bounded on $[0,t_r]$. Then, the function $$ \omega(\cdot):= \frac{1}{B^*\varphi}\bigg(-\lambda\sigma(\cdot) -\alpha\vert\sigma(\cdot)\vert^{\frac{1}{2}} \mathrm{sign}(\sigma(\cdot))+w(\cdot)\bigg)$$ is also bounded on $[0,t_r]$. Therefore, there exists $C_1>0$ such that
\begin{equation}\label{omega_borne}
     \Vert \omega\Vert_{L^2(0,t_r)}\leq C_1 t_r^{\frac{1}{2}}.
\end{equation}
Now, according to \cite[Theorem 2]{moreno2012strict}, there exist positive constants $C_2$, $C_3$  such that
\begin{equation} \label{tr} \left\{\begin{array}{l}
t_r<C_2\left(\vert \sigma(0)\vert +\vert w_0 \vert\right), \\
\vert w(t)\vert\leq C_3 \vert w_0\vert.
\end{array}\right.
\end{equation} Using  Cauchy-Schwarz's inequality, we obtain
\begin{equation} \label{sigma0}
    \vert \sigma(0)\vert=\vert \langle \varphi, z_0\rangle_H\vert\leq \Vert \varphi\Vert_H\Vert z_0\Vert_H.
\end{equation}
As a consequence, according to \eqref{inegalite_stabilite_twisting}, \eqref{omega_borne}, \eqref{tr} and \eqref{sigma0}, there exists $C_4>0$  such that, for all $t\in [0,t_r]$,
\begin{equation}\label{inegalite_stabilite_twisting_1}
\Vert z(t) \Vert_{H}+\vert w(t)\vert\leq C_4\bigg(\Vert z_0 \Vert_{H}+\vert w_0\vert
   +\sqrt{\Vert z_0 \Vert_{H}+\vert w_0\vert}\bigg).
\end{equation} 

This concludes the proof of Lyapunov stability, in the sense given in \cite[Definition 2.3]{prieur}, of  system \eqref{system-final_super_twisting} over the time interval $[0,t_r]$. 

\hfill $\Box$

\section{Illustrative example: Heat equation}
\label{sec:exemple}
Consider the following system,
\begin{equation}
\label{heat equation}
\left\{
\begin{split}
&z_{t}(t,x) = z_{xx}(t,x),\quad (t,x)\in \mathbb{R}_{\geq 0}\times [0,1],\\
&z_x(t,0)=c_0z(t,0),\quad t\in \mathbb{R}_+,\\
&z_x(t,1)=u(t)+d(t),\quad t\in \mathbb{R}_+,\\
&z(0,x)=z_0(x),
\end{split}
\right.
\end{equation}
where $c_0$ is a positive constant,  $u(t)\in \R$ is the control input and $d(t)\in\R$  is an unknown disturbance.

This equation can be written in an abstract way as in \eqref{system_depart} if one sets $H=L^2(0,L)$,
\begin{equation}
\begin{split} 
A:D(A)\subset L^2(0,L)&\rightarrow L^2(0,L),\\
z&\mapsto z^{\prime\prime},
\end{split}
\end{equation}
where 
\begin{equation}
D(A) :=\lbrace z\in \mathcal{H}^2(0,1)\mid z'(0)=c_0z(0);z^\prime(1)=0\rbrace,
\end{equation}
and the control operator $B$ is  the delta function in
$\mathcal{L}(\R,D(A)')$ defined as follow  
\begin{equation}
    \langle \varphi,Bu\rangle_{D(A),D(A)'}=\varphi(1)u
\end{equation}  for all $u\in\R$ and $\varphi \in D(A)$, where $\langle\cdot,\cdot\rangle_{D(A),D(A)'}$ is the dual product. The adjoint operator of $A$ is
\begin{equation}
\begin{split}
A^*: D(A^*)\subset H&\rightarrow H,\\
z & \mapsto  z^{\prime\prime},
\end{split}
\end{equation}
with $D(A^*):= \lbrace z\in \mathcal{H}^2(0,1)\mid z'(0)=c_0z(0);z^\prime(1)=0\rbrace$. It can be checked that the operator $A$ is self-adjoint in $H$. The adjoint of operator of $B$ is 
\begin{equation}
\begin{split}
B^* : D(A^*)&\rightarrow \R\\
\varphi &\mapsto \varphi(1).
\end{split}
\end{equation}
According to \cite[Lemma 2.1 and 2.2]{liu2015active}, $A$ generates a strongly continuous semigroup $(\mathbb{T}(t))_{t\geq0}$ of contractions on $H$ and the operator $B$ is admissible for the semigroup $(\mathbb{T}(t))_{t\geq0}$. Thus, the operators $A$ and $B$ satisfy  the items (i) and (ii) Assumption \ref{assump1}. Moreover, according to \cite[Lemma 2.1]{liu2015active}, the origin of
\begin{equation}
\label{heat_equation}
\left\{
\begin{split}
&z_{t}(t,x) = z_{xx}(t,x),\quad (t,x)\in \mathbb{R}_{\geq 0}\times [0,1],\\
&z_x(t,0)=c_0z(t,0),\quad t\in \mathbb{R}_+,\\
&z_x(t,1)=0,\quad t\in \mathbb{R}_+,\\
&z(0,x)=z_0(x),
\end{split}
\right.
\end{equation}is globally exponentially stable in $H$. As a consequence, Item (iv) of Assumption \ref{assump1}  holds for the operator $L$ equal to the zero operator. 

Since $A$ is self-adjoint, then its  spectrum is real. Therefore, a direct computation gives that the eigenpairs $(\lambda,\varphi_\lambda)$ of $A$ satisfies
\begin{equation}
    \left\{
\begin{split}
&\varphi_\lambda(x) = \cos (\sqrt{-\lambda}x)+\frac{c_0}{\sqrt{-\lambda}}\sin (\sqrt{-\lambda}x),\\
&\sqrt{-\lambda}\tan (\sqrt{-\lambda})=c_0.
\end{split}
\right.
\end{equation}
The function $x\in\R\setminus\{\frac{\pi}{2}+k\pi;k\in \mathbb{Z}\}\mapsto \tan(x)$ is surjective. Thus, the equation $\sqrt{-\lambda}\tan (\sqrt{-\lambda})=c_0$ admits a solution.  Note that $\lambda$ is negative, since the origin of \eqref{heat_equation} is globally exponentially stable in $H$.

Let $\varphi_\lambda\in D(A)$ the eigenfunction of the operator $A$ associated to $\lambda$. The sliding variable and the feedback law under consideration are as follows
\begin{align}\label{control_sigmal_charleur}\notag
&\sigma(t)=\int^L_0 z(t,x)\varphi_\lambda(x) dx\quad\mbox{and}\\
    &u(t)=  -\frac{1}{\varphi_\lambda(1)}\big(\lambda\sigma(t)+K\mathrm{sign}(\sigma(t))\big).
\end{align}
Thus, if we choose $d$ and $K$ as in Assumption \ref{assump2} and \ref{assump3}, we can conclude  that the origin of
\begin{equation}
\label{heat equation_1}
\left\{
\begin{split}
&z_{t}(t,x) = z_{xx}(t,x),\quad (t,x)\in \mathbb{R}_{\geq 0}\times [0,1],\\
&z_x(t,0)=c_0z(t,0),\quad t\in \mathbb{R}_+,\\
&z_x(t,1)\in-\frac{1}{\varphi_\lambda(1)}\big(\lambda\sigma(t)+K\mathrm{sign}(\sigma(t))\big)+d(t),\\&\quad t\in \mathbb{R}_+,\\
&z(0,x)=z_0(x),
\end{split}
\right.
\end{equation}is globally asymptotically stable in $H$.
On the other hand the super-twisting control under consideration is as follows 
 \begin{equation}\label{super twisting control charleur} 
\left\{
\begin{aligned}
& u(t) = \frac{1}{\varphi_\lambda(1)}\bigg(-\lambda\sigma(t) -\alpha\vert\sigma(t)\vert^{\frac{1}{2}} \mathrm{sign}(\sigma(t))+v(t)\bigg),\\
&\dot v(t)\in-\beta\mathrm{sign}(\sigma(t)).
\end{aligned}
\right.
\end{equation}
Therefore,  if we choose $d$ as in Assumption \ref{assump4}, $\beta$ and $\alpha$ as in \eqref{parametre} we can conclude  that the origin of
\begin{equation}
\label{heat equation_2}
\left\{
\begin{split}
&z_{t}(t,x) = z_{xx}(t,x),\quad (t,x)\in \mathbb{R}_{\geq 0}\times [0,1],\\
&z_x(t,0)=c_0z(t,0),\quad t\in \mathbb{R}_+,\\
&z_x(t,1)=\frac{1}{\varphi_\lambda(1)}\bigg(-\lambda\sigma(t) -\alpha\vert\sigma(t)\vert^{\frac{1}{2}} \mathrm{sign}(\sigma(t))+v(t)\bigg)\\&+d(t),\quad t\in \mathbb{R}_+,\\
&z(0,x)=z_0(x),
\end{split}
\right.
\end{equation}is globally asymptotically stable in $H$.

Using the finite difference method \cite{li2017numerical}, we performed some numerical simulations. We choose $\lambda=-2c_0- \pi^2$ which is an approximated solution of $\sqrt{-\lambda}\tan (\sqrt{-\lambda})=c_0$, $c_0 = 0.5$, $K = 2.5$, $z_0(x) = 10x^3$ and $d(t) = 2 \sin{(t)}$. The space and time steps are taken as 0.1 and 0.0001, respectively.
\begin{figure}
  \begin{center}
  \includegraphics[scale=0.5]{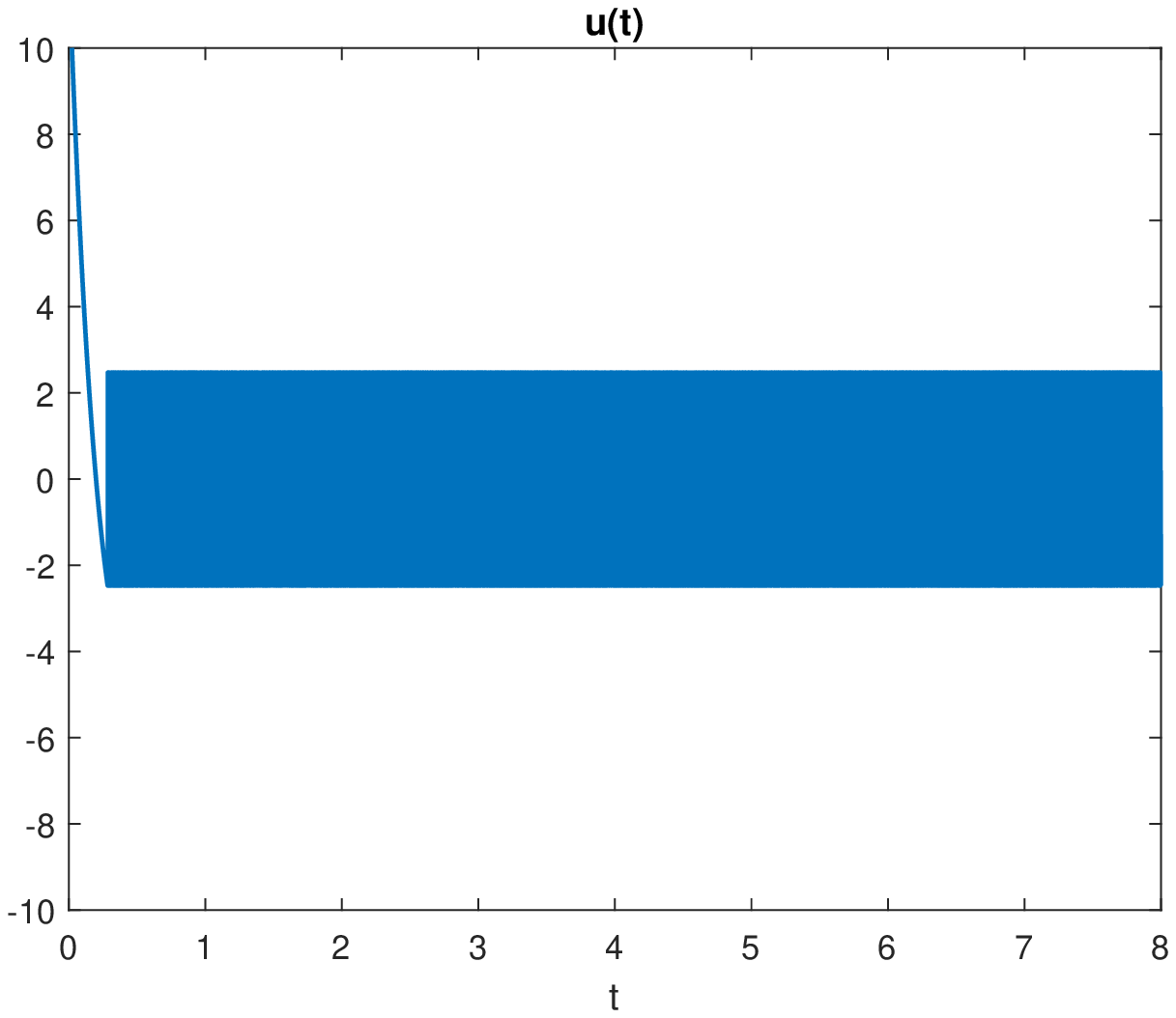}
  \includegraphics[scale=0.5]{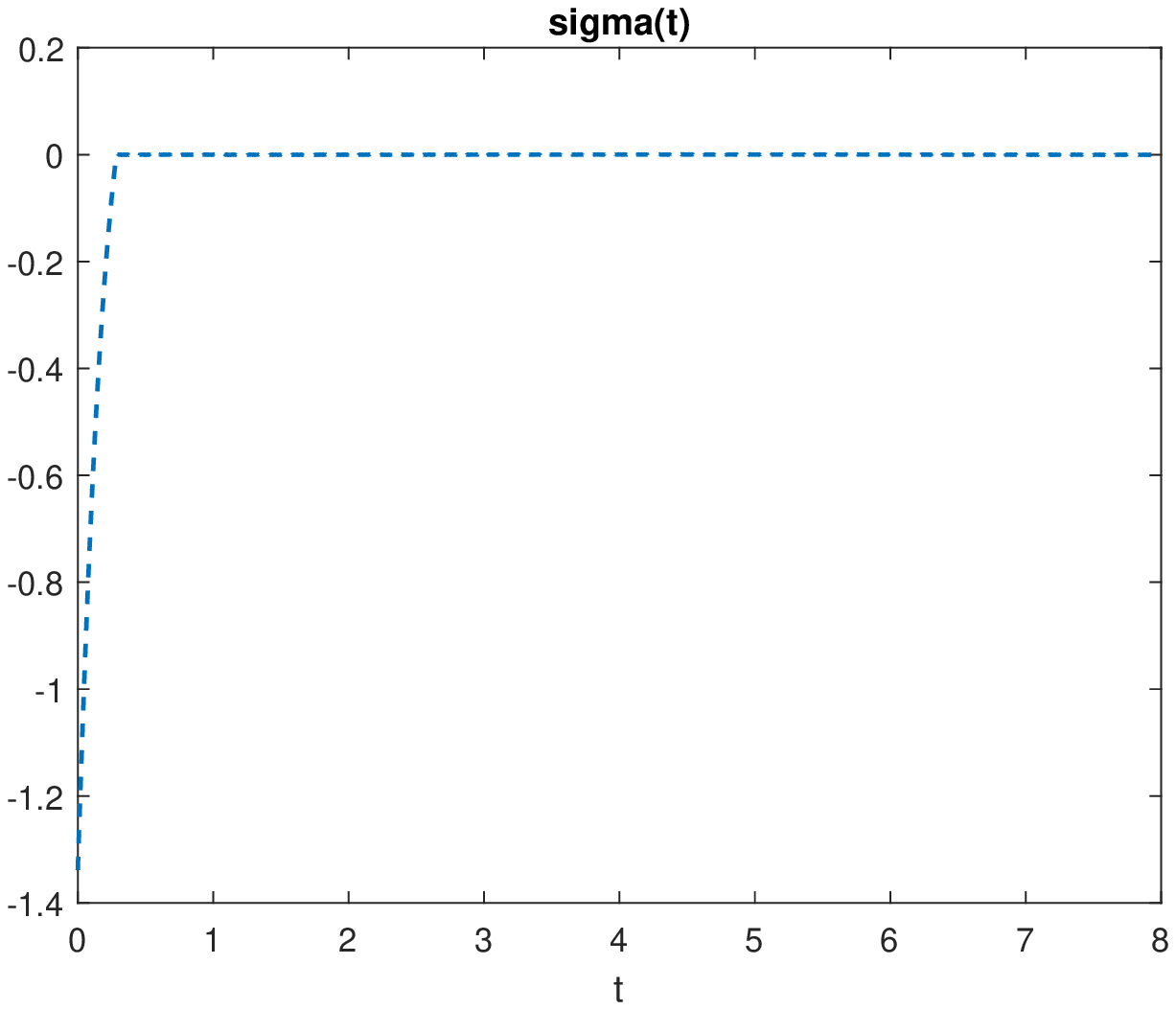}
    \includegraphics[scale=0.5]{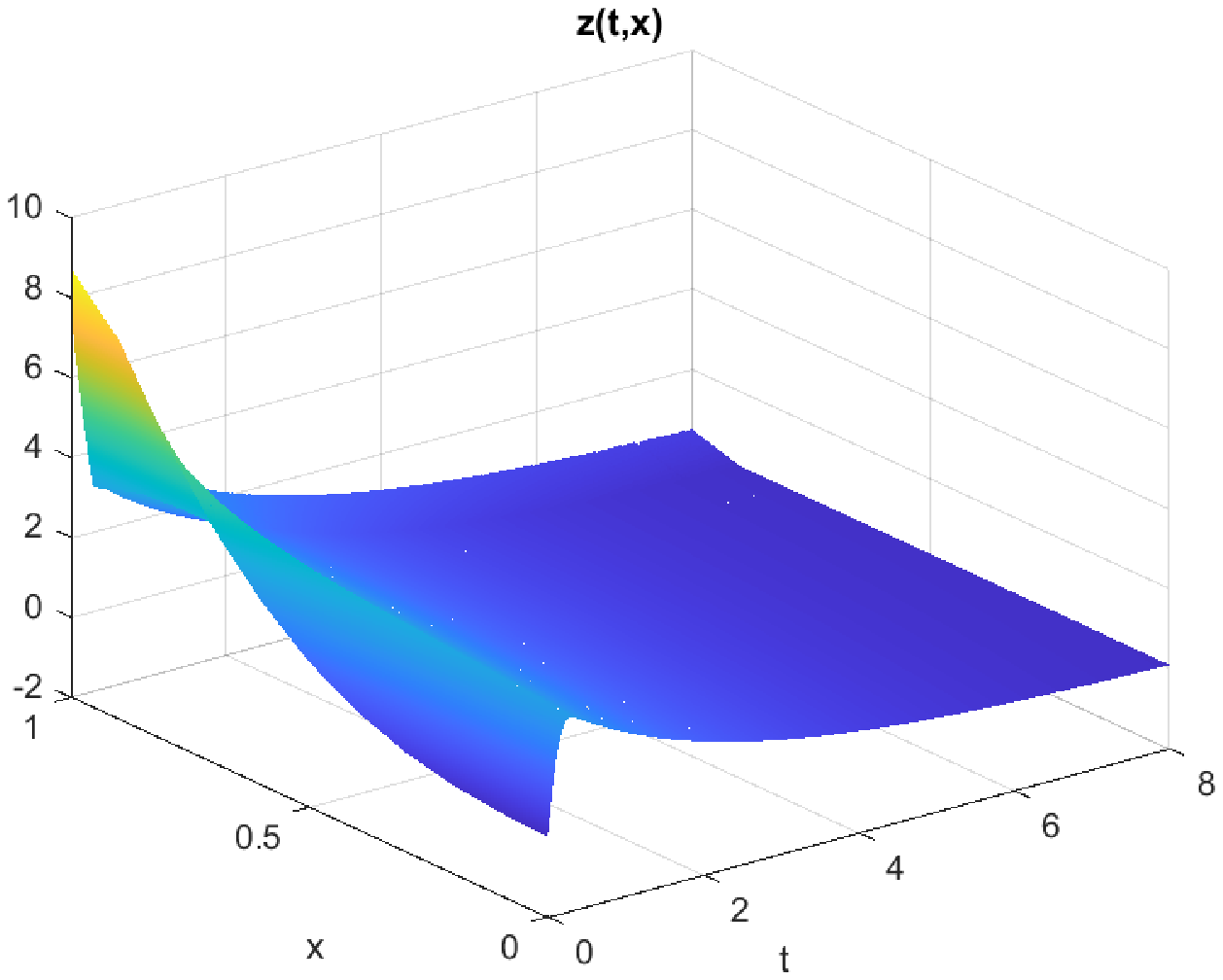}
  \end{center}
  \caption{{\bf Top.} Control input $u(t)$ versus time $t$ (sec). {\bf Middle.} Sliding variable $\sigma(t)$ versus time $t$ (sec). {\bf Bottom.} Solution $z$  versus time $t$ (sec) and position $x$.}
    \label{fig:control_charleur}
\end{figure}

In Figure \ref{fig:control_charleur}- {\bf Top.}  the control input $u$ defined in \eqref{control_sigmal_charleur} makes chattering phenomenon appearing once the sliding variable has converged (see Figure \ref{fig:control_charleur}-{\bf Middle.}). In Figure \ref{fig:control_charleur}-{\bf Bottom.}, the  stabilization of $z$ of \eqref{heat equation_1} is illustrated.
 
 \begin{figure}
  \begin{center}
  \includegraphics[scale=0.5]{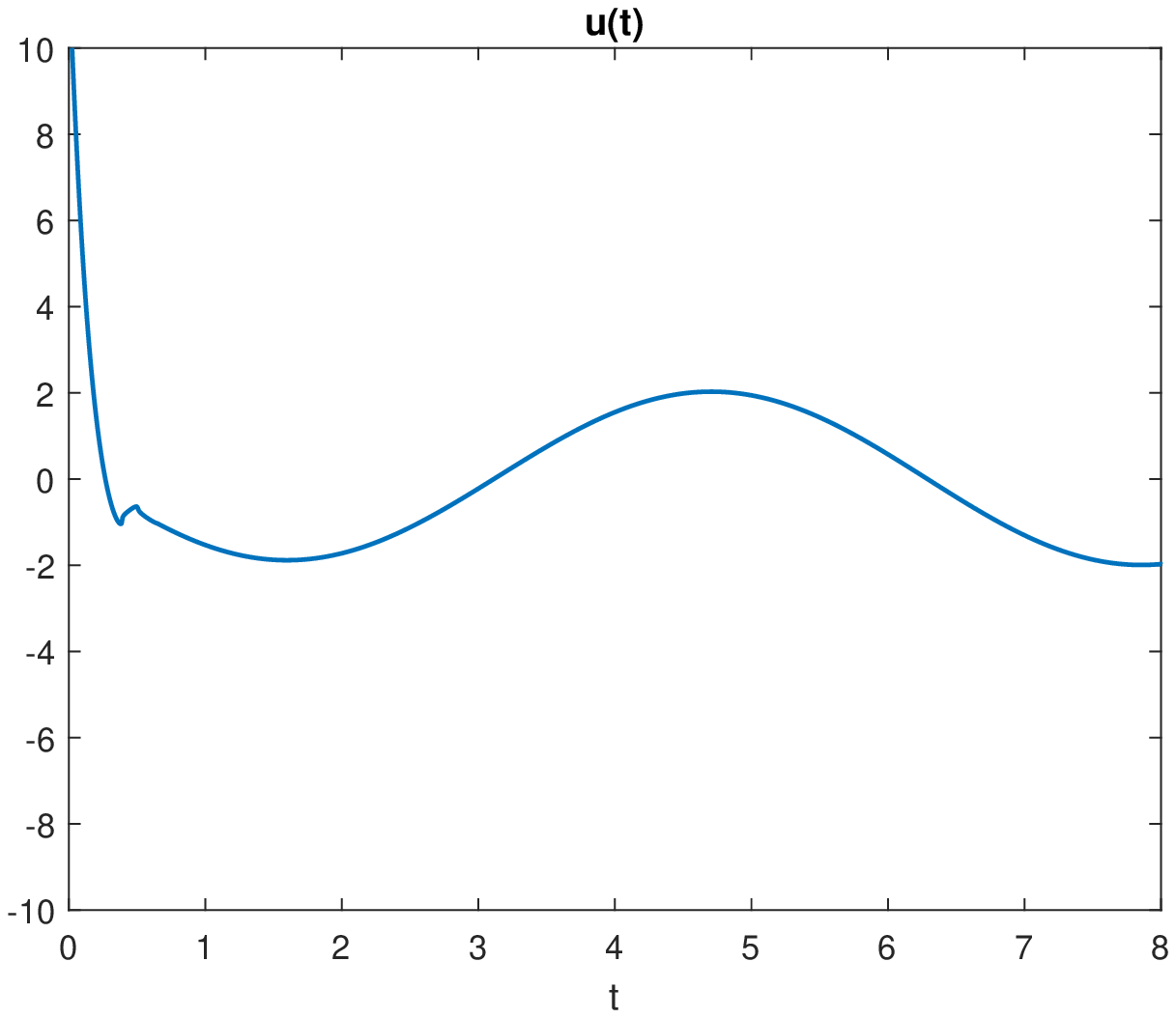}
  \includegraphics[scale=0.5]{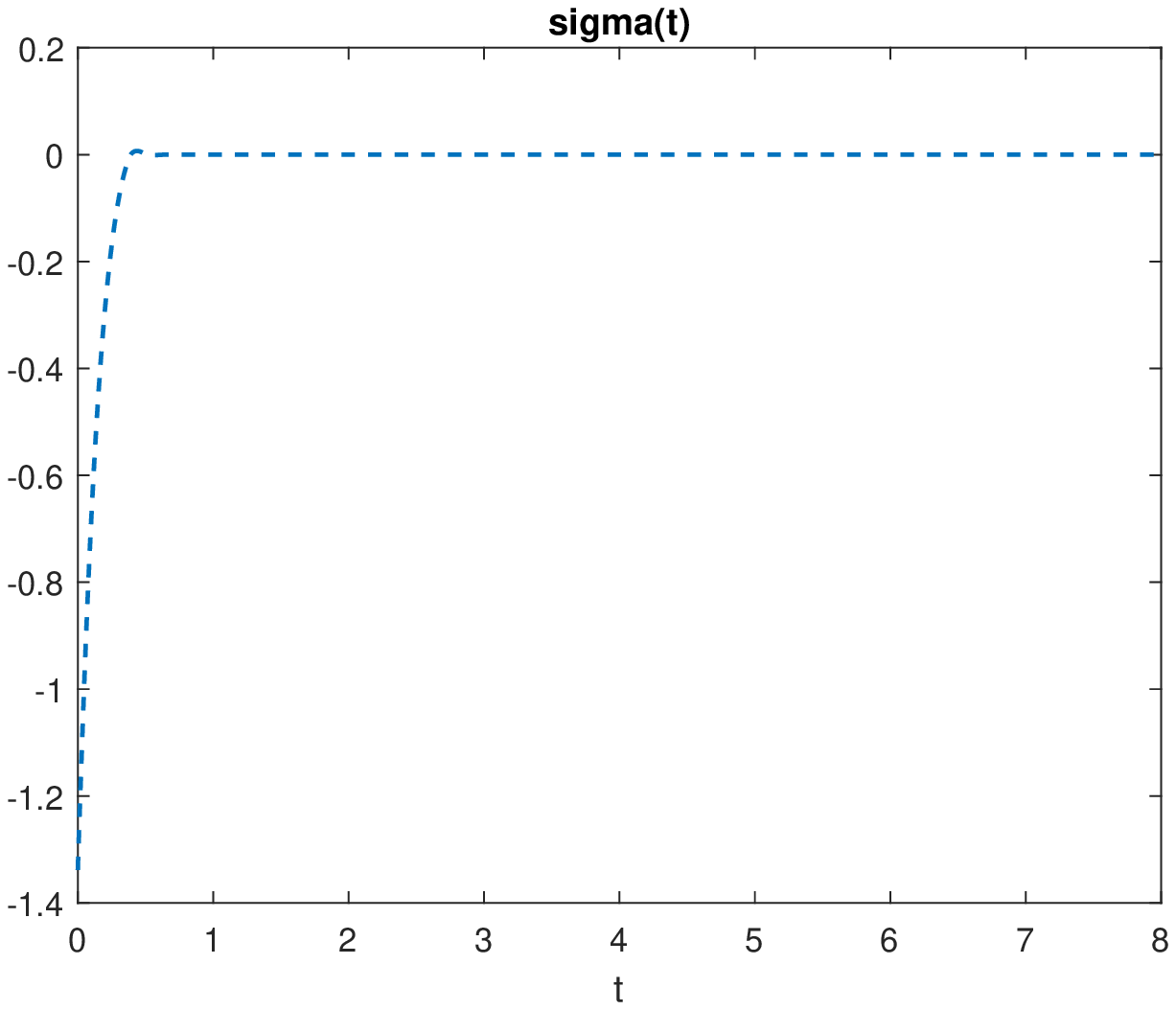}
    \includegraphics[scale=0.5]{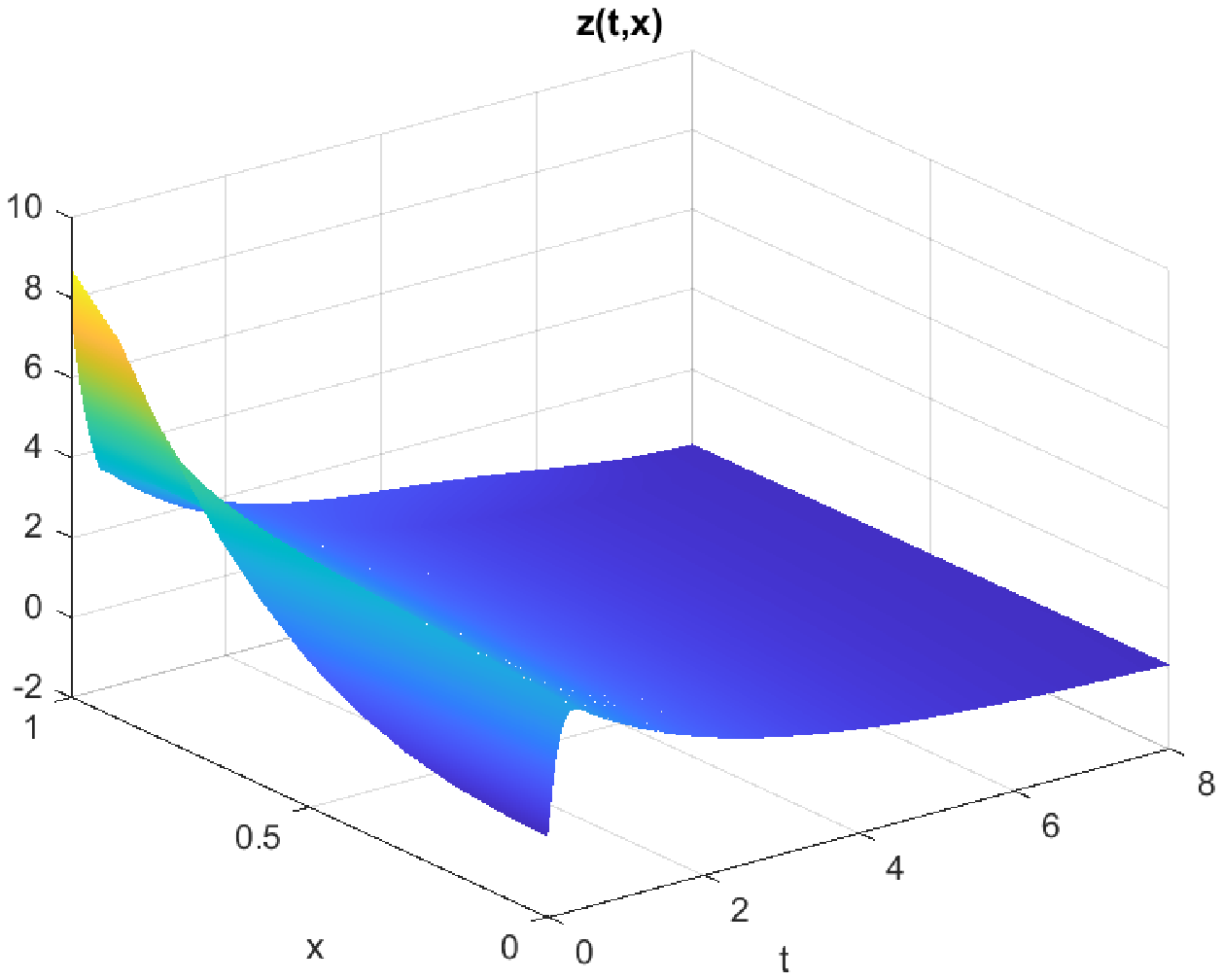}
  \end{center}
  \caption{{\bf Top.} Control input $u(t)$ versus time $t$ (sec). {\bf Middle.} Sliding variable $\sigma(t)$ versus time $t$ (sec). {\bf Bottom.} Solution $z$  versus time $t$ (sec) and position $x$.}
    \label{fig:control_super_twisting}
\end{figure}

Figures \ref{fig:control_super_twisting}  is obtained with the same settings as the Figures \ref{fig:control_charleur} with $\beta=2.5$ and $\alpha=2.2$. It must be noted that, thanks to the use of super twisting algorithm, the chattering on $u$ has been removed (super twisting is continuous) whereas the stabilization is kept.

\section{Conclusion}
\label{sec_conclusion}

In this paper, we have proposed a design method based on sliding mode control  for the stabilization of class of linear infinite-dimensional systems with unbounded control operators and subject to a boundary disturbance. The existence of solutions of the closed-loop system has been proved as well as the disturbance rejection and the asymptotic stability  of the  closed-loop control system. We further have extended the super-twisting method for the same class of linear infinite-dimensional systems.

Future works will consider the case  where the  operator $A$ in \eqref{system_depart} is nonlinear, for which many notions will need to be adapted such as the controllability or the admissibility. It might also be interesting to investigate the case where the disturbance does not match with the control as it has been done for ODEs in \cite{castanos2006analysis}.

\bibliographystyle{abbrv}
\bibliography{biblio}

\end{document}